\documentclass[10pt,a4paper]{article}
\usepackage{amsmath}
\usepackage{latexsym,amssymb,amsmath,amsthm,amsfonts}
\usepackage{graphicx}
\usepackage{enumerate}

\newtheorem{lemma}{Lemma}[section]

\newtheorem{theorem}[lemma]{Theorem}
\newtheorem{corollary}[lemma]{Corollary}

\newtheorem{example}[lemma]{Example}

\numberwithin{equation}{section}

\newcommand{\ud}{\mathrm{d}}
\newcommand{\RR}{\mathbb{R}}
\newcommand{\f}{\frac}

\newcommand{\xx}{|x|^2}
\newcommand{\yy}{|y|^2}
\newcommand{\xy}{\langle x,y\rangle}

\newcommand{\pp}[2]{\frac{\partial{#1}}{\partial{#2}}}

\newcommand{\pppp}[4]%
  {\frac{\partial^3{#1}}{\partial{#2}\partial{#3}\partial{#4}}}

\newcommand{\p}{\phi}

\newcommand{\pbs}{\alpha\phi\left(b^2,s\right)}
\newcommand{\gab}{\alpha\phi\left(b^2,\frac{\beta}{\alpha}\right)}
\newcommand{\pt}{\phi_2}
\newcommand{\po}{\phi_1}
\newcommand{\ptt}{\phi_{22}}
\newcommand{\pot}{\phi_{12}}

\renewcommand{\a}{\alpha}
\renewcommand{\b}{\beta}
\newcommand{\ab}{(\alpha,\beta)}

\newcommand{\ba}{\bar\alpha}
\newcommand{\bb}{\bar\beta}

\newcommand{\aij}{a_{ij}}

\newcommand{\bi}{b_i}
\newcommand{\bj}{b_j}

\newcommand{\bij}{b_{i|j}}

\newcommand{\baij}{\bar a_{ij}}

\newcommand{\bbij}{\bar b_{i|j}}

\newcommand{\G}{{}^\alpha G^i}

\newcommand{\bG}{{}^{\bar\alpha}G^i}

\begin{document}
\title{Projectively flat general $\ab$-metrics with constant flag curvature}

\author{Changtao Yu and Hongmei Zhu}

\date{January 17, 2015}
\maketitle

\begin{abstract}
In this paper we study the flag curvature of a new class of Finsler metrics called general $\ab$-metrics, which are defined by a Riemannian metric $\a$ and a $1$-form $\b$. The classification of such metrics with constant flag curvature are completely determined under some suitable conditions, which make them locally projectively flat. As a result, we construct many new projectively flat Finsler metrics with flag curvature $1$, $0$ and $-1$ in Section 9, all of which are of singularity at some directions. The simplest one is given by $F=\frac{(b\a+\b)^2}{\a}$ where $b=\|\b\|_\a$.
\end{abstract}

\section{Introduction}
In Finsler geometry, many important Finsler metrics with constant flag curvature are locally projectively flat. For example, the generalized Funk metrics
\begin{eqnarray}\label{funk}
F=\f{\sqrt{(1-\xx)\yy+\xy^2}}{1-\xx}\pm\left\{\f{\xy}{1-\xx}+\f{\langle a,y\rangle}{1+\langle a,x\rangle}\right\}
\end{eqnarray}
are locally projectively flat with constant flag curvature $K=-\frac{1}{4}$, where $a$ is a constant vector\cite{szm-pfrm}. (\ref{funk}) belong to a special class of Finsler metrics called Randers metrics given in the form $F=\a+\b$, where $\a$ is a Riemannian metric and $\b$ is a $1$-form. Moreover, the generalized Berwald's metrics
\begin{eqnarray}\label{berwald}
F=\f{((1+\langle a,x\rangle)(\sqrt{(1-\xx)\yy+\xy^2}+\xy)+(1-\xx)\langle a,y\rangle)^2}{(1-\xx)^2\sqrt{(1-\xx)\yy+\xy^2}}
\end{eqnarray}
are also locally projectively flat with constant flag curvature $K=0$\cite{szm-ygc}. (\ref{berwald}) belong to the so-called square metrics given in the form $F=\frac{(\a+\b)^2}{\a}$.\cite{szm-yct-oesm}

Both Randers metrics and square metrics belong to the metrical category called $\ab$-metrics, which are given in the form $F=\a\phi(\frac{\b}{\a})$, where $\phi(s)$ is a smooth function. In 2007, Li-Shen proved that except Riemannian metrics and locally Minkowskian metrics, any locally projectively flat $\ab$-metric with constant flag curvature $K$ is either locally isometric to a generalized Funk metric after a scaling when $K<0$, or locally isometric to a generalized Berwald's metric after a scaling when $K=0$\cite{LB}.

Randers metrics can be expressed in another famous form
\begin{eqnarray}\label{NP}
F=\frac{\sqrt{(1-b^2)\alpha^2+\beta^2 }}{1-b^2} +\frac{\beta}{1-b^2},
\end{eqnarray}
where $b:= \|\beta_x\|_{\alpha}$ is the length of $\b$. Combining with Bao-Robles-Shen's well-known classification result\cite{db-robl-szm-zerm} and the related discussions in \cite{szm-pfrm}, one can see that a Randers metric is locally projectively flat and of constant flag curvature if and only if $\a$ in (\ref{NP}) is locally projectively flat and $\b$ is closed and homothetic with respect to $\a$. Note that Beltrami's theorem says that a Riemannian metric is locally projectively flat if and only if it is of constant sectional curvature. So the above fact means that $\a$ and $\b$ satisfy
\begin{equation*}
{}^\alpha R^i{}_j=\mu(\a^2\delta^i{}_j-y^iy_j),\qquad\bij=c  \aij,
\end{equation*}
where $c$ is a constant.

Square metrics can also be expressed in another form
\begin{equation}\label{square}
F=\frac{(\sqrt{(1-b^2)\alpha^2+\beta^2}+\beta)^2}{(1-b^2)^2\sqrt{(1-b^2)\alpha^2+\beta^2}}.
\end{equation}
The first author showed that a square metric is locally projectively flat if and only if $\a$ in (\ref{square}) is locally projectively flat and $\b$ is closed and conformal with respect to $\a$\cite{yct-dhfp}, i.e.,
\begin{eqnarray}\label{condition}
{}^\alpha R^i{}_j=\mu(\a^2\delta^i{}_j-y^iy_j),\qquad\bij=c(x)\aij,
\end{eqnarray}
where $c(x)$ is a scalar function on the manifold. Later on, Z. Shen and the first author proved that a square metric is an Einstein metric if and only if $\a$ and $\b$ satisfy
\begin{equation*}
{}^\alpha Ric=0,\qquad\bij=c\aij,
\end{equation*}
where $c$ is a constant\cite{szm-yct-oesm}.

One can see from the above facts that the expressions (\ref{NP}) and (\ref{square}) have the advantage of clearly illuminating
the underlying geometry, although they are more complicated in algebraic form. This is a common phenomenon about $\ab$-metrics\cite{yct-odfa,yct-odfr}. Actually, both Randers metrics and square metrics belong to a larger class of Finsler metrics called general $\ab$-metrics, which are also defined by a Riemannian metric $\a$ and a $1$-form $\b$ and given in the form
\begin{eqnarray}\label{generalab}
F=\gab,
\end{eqnarray}
where $\phi(b^2,s)$ is a smooth function\cite{yct-zhm-onan}.

If $\phi=\phi(s)$ is independent of $b^2$, then $F=\a\phi(\frac{\b}{\a})$ is a $\ab$-metric. If $\a=|y|$, $\b=\xy$, then $F=|y|\phi(|x|^2,\frac{\xy}{|y|})$ is the so-called spherically symmetric Finsler metrics. Moreover, general $\ab$-metrics include part of Bryant's metrics and part of fourth root metrics. That is to say, general $\ab$-metrics make up of a much large class of Finsler metrics, which makes it possible to find out more Finsler metrics to be of great properties. For example, in $\ab$-metrics we cann't find out any non-Ricci flat Einstein metric unless it is of Randers type\cite{cst}. The main reason is that the category of $\ab$-metrics is a little small. If we search Einstein metrics in general $\ab$-metrics, then it is not hard to find out metrics with positive and negative Ricci constant\cite{szm-yct-oefm}.

Come back to our discussions. It is clear that the corresponding functions $\phi(b^2,s)$ of (\ref{NP}) and (\ref{square}) are given by $\phi=\frac{\sqrt{1-b^2+s^2}}{1-b^2} + \frac{s}{1-b^2}$ and $\phi=\frac{(\sqrt{1-b^2+s^2}+s)^2 }{(1-b^2)^2 \sqrt{1-b^2 +s^2 }}$ respectively. Moreover, both of them satisfy the following PDE:
\begin{eqnarray}\label{pde}
\phi_{22}=2(\phi_1-s\phi_{12}).
\end{eqnarray}
Here  $\phi_1$ means the derivation of $\phi$ with respect to the first variable $b^2$.

In fact, the first author proved that when dimension $n\geq3$, every non-trivial locally projectively flat $\ab$ metric can be reexpressed as a new form $F=\pbs$ such that the corresponding function $\phi$ satisfies (\ref{pde}), and at the same time $\a$ and $\b$ satisfy (\ref{condition})\cite{yct-dhfp}. We believe that it also holds for general $\ab$-metrics, although we don't still know how to prove it by now. Until now, it is known that if (\ref{condition}) holds, then the general $\ab$-metric $F=\pbs$ is locally projectively flat if and only if $\phi$ satisfies (\ref{pde})\cite{szm-yct-oefm}.

The aim of this paper is to study general $\ab$-metrics with constant flag curvature. It's worth mentioning here that L. Zhou proved an interesting result in 2010: if a square metric is of constant flag curvature, then it must be locally projectively flat\cite{zlf}. It is not true for Randers metrics, because there are many Randers metrics with constant flag curvature which are not locally projectively flat actually\cite{db-robl-szm-zerm}. Even so, we have reason to believe that Zhou's result holds for any non-Randers type general $\ab$-metrics. More specifically, we conjecture that except Randers metrics, there does not exist any non locally projectively flat regular general $\ab$-metric to be of constant flag curvature. Randers metrics are very particular, the key reason is that any Randers metric will still turn to be a Randers metric after navigation transformation\cite{db-robl-szm-zerm}, but for a non-Randers type general $\ab$-metric, it will not turn to be a general $\ab$-metric after navigation transformation in general\cite{yct-zhm-onan}.

Hence, we will discuss our problem under the assumption that $\a$ and $\b$ satisfy the conditions (\ref{condition}), and $\phi$ satisfies the condition (\ref{pde}). In this case, the corresponding general $\ab$-metric must be locally projectively flat. Moreover, Lemma \ref{c(x)} shows that the conformal factor $c(x)$ in (\ref{condition}) must satisfy $c^2=\kappa-\mu b^2$ for some constant $\kappa$.

To be more clear, let's illustrate our assumption in this paper again:

{\bf Assumption:} $\a$, $\b$ and $\phi$ satisfy
\begin{eqnarray}\label{conditions}
{}^\alpha R^i{}_j=\mu(\a^2\delta^i{}_j-y^iy_j),\qquad\bij=(\kappa-\mu b^2)\aij,\qquad\phi_{22}=2(\phi_1-s\phi_{12})
\end{eqnarray}
respectively.

We believe that such conditions are natural. Our main reason is that all the known general $\ab$-metrics with constant flag curvature, including Bryant's metrics which are not discussed above, can be reexpressed to fit it.

The general $\ab$-metrics $F=\a\phi(b^2,\frac{\b}{\a})$ with constant flag curvature under our assumption can be completely solved. Firstly, we have the following equivalent characterization.

\begin{theorem}\label{main1}
Let $F=\a\phi(b^2,\frac{\b}{\a})$ be a general $\ab$-metric on an $n$-dimensional manifold $M$ with $n\geq3$, where $\a$, $\b$ and $\p$ satisfy (\ref{conditions}). Then $F$ is of constant flag curvature $K$ if and only if the function $\p=\p(b^2,s)$ satisfies the following PDE:
\begin{eqnarray}\label{pde2}
( \kappa -\mu b^2) \left[\psi^2-(\psi_2+2s\psi_1)\right]+\mu s \psi + \mu =K\p^2,
\end{eqnarray}
where $\psi:=\frac{\pt+2s\po}{2\p}$.
\end{theorem}

The Riemannian metrics $\a$ and $1$-forms $\b$ satisfying (\ref{condition}) have already been determined completely~(see (\ref{ab})). According to Theorem \ref{main1}, in order to determine the general $\ab$-metrics with constant flag curvature under our assumption, we only need to solve Equation (\ref{pde}) and (\ref{pde2}).

The case when $\kappa=0$ and $\mu=0$ is trivial, because in the case $\a$ is locally Euclidian and $\b$ is parallel with respect to $\a$. As a result, $F=\a\phi(b^2,\frac{\b}{\a})$ is locally Minkowskian and hence flat automatically for any suitable function $\phi(b^2,s)$.

When $\kappa\neq0$ and $\mu=0$, we have the following result.

\begin{theorem}\label{main2} Let $F=\a\phi(b^2,\frac{\b}{\a})$ be a general $\ab$-metric on an $n$-dimensional manifold $M$ with $n\geq3$, where $\a$, $\b$ and $\phi$ satisfy (\ref{conditions}) with $\kappa\neq0$ and $\mu=0$. Then $F$ is of constant flag curvature $K$ if and only if $\phi$ is given by one of the forms:
\begin{eqnarray}
\phi & = & \frac{1}{2\sqrt{-\sigma}} \frac{1}{\sqrt{C-b^2+s^2}\pm s},\label{solution1}\\
\phi & = & \frac{q(u)}{q^2(u) (D q(u)+ v)^2+\sigma},\label{solution2}
\end{eqnarray}
where $\sigma: = K/\kappa$, $u:=b^2-s^2$ and $v:=s$, the function $q(u)$ satisfies the following equation:
$$D^2  q^4+(u-C)q^2 -\sigma=0,$$
where $C$ and $D$ are constants.
\end{theorem}

The case when $\kappa\neq0$ and $\mu\neq0$ can be reduced to the above case by some special deformations. See Section 3 and Example \ref{ex5} for details.

The case when $\kappa=0$ and $\mu\neq0$ is very special, and we have the following result.

\begin{theorem}\label{main3} Let $F=\a\phi(b^2,\frac{\b}{\a})$ be a general $\ab$-metric on an $n$-dimensional manifold $M$ with $n\geq3$, where $\a$,  $\b$ and $\phi$ satisfy (\ref{conditions}) with $\kappa=0$ and $\mu\neq0$. Then $F$ is of constant flag curvature $K$ if and only if $\phi$ is given by:
\begin{eqnarray}
\phi(u,v)=\frac{2q(u)(\sqrt{u+v^{2}}\pm v)^{2}}{[q(u)(\sqrt{u+v^{2}}\pm v)^{2}+p(u)]^2+\tau},\label{solution3}
\end{eqnarray}
where $\tau: = -K/\mu$, $u:=b^2-s^2$ and $v:=s$, the functions $p(u)$ and $q(u)$ are given by one of the forms:
\begin{eqnarray}\label{pppp}
p(u)=\pm\sqrt{-\tau},\qquad q(u)=\pm\f{(C\pm\sqrt{C^2+8pu})^2}{4u^2}
\end{eqnarray}
or
\begin{eqnarray}
p(u)&=&\pm\sqrt{\f{-(C^2-D)\tau-C(C\tau-2u)\pm\sqrt{D(C\tau-2u)^2-D(C^2-D)\tau^2}}{2(C^2-D)}},\label{pp}\\
q(u)&=&\f{p^2+\tau-upp'\pm\sqrt{(p^2+\tau-upp')^2-(p^2+\tau)u^2p'}}{u^2p'},\label{qq}
\end{eqnarray}
where $C$ and $D$ are constants.
\end{theorem}

By Theorem \ref{main3}, we can obtain some new Finsler metrics with constant flag curvature $1$, $0$ and $-1$. For example, it is easy to check that $\phi(b^2,s)=(b+s)^2$ satisfies Equations (\ref{pde}) and (\ref{pde2}) with $\kappa=0$, $\mu\neq0$ and $K=0$, so $$F=\f{(b\a+\b)^2}{\a}$$
is projectively flat and of vanishing flag curvature,
where
\begin{eqnarray*}
\a=\f{\sqrt{(1+\mu|x|^2)|y|^2-\mu\langle x,y\rangle^2}}{1+\mu|x|^2},\qquad
\b=\f{\lambda\langle x,y\rangle+(1+\mu|x|^2)\langle a,y\rangle-\mu\langle a,x\rangle\langle x,y\rangle}{(1+\mu|x|^2)^\frac{3}{2}}.
\end{eqnarray*}
with additionally $\lambda^2+\mu|a|^2=0$, which makes $\a$ and $\b$ satisfy (\ref{conditions}) with $\kappa=0$.  One can find more examples in Section 9.

Notice that $\p-s\pt=b^2-s^2$ and $\p-s\pt+(b^2-s^2)\ptt=3(b^2-s^2)$, so such metrics are non-regular at the directions $(y^i)=\pm(b^i)$. Moreover, $\phi=0$ when $(y^i)=-(b^i)$. Actually, all the metrics determined by Theorem \ref{main3} have the same singularity. Hence, we have
\begin{corollary}
When $n\geq3$, all the non-trivial regular general $\ab$-metrics $F=\a\phi(b^2,\frac{\b}{\a})$ with constant flag curvature satisfying (\ref{conditions}) are completely determined by Theorem \ref{main2}.
\end{corollary}

\section{Preliminaries}
Let $F$ be a Finsler metric on an $n$-dimensional manifold $M$ and $G^{i}$ be the geodesic coefficients of $F$, which are defined by
\begin{eqnarray*}
G^{i}=\frac{1}{4}g^{il}\left\{[F^{2}]_{x^{k}y^{l}}y^{k}-[F^{2}]_{x^{l}}\right\},
\end{eqnarray*}
where $(g^{ij}):=\left(\frac{1}{2}[F^{2}]_{y^{i}y^{j}}\right)^{-1}$. For a Riemannian metric, the spray coefficients are determined by its Christoffel symbols as $G^{i}(x,y)=\frac{1}{2}\Gamma^{i}_{jk}(x)y^{j}y^{k}$.

For any $x\in M$ and $y\in T_{x}M\backslash\{0\}$, the Riemann
curvature tensor $R_{y}=R^{i}{}_{j}\frac{\partial}{\partial x^{i}}\otimes
dx^{j}$ of $F$ is defined by
$$ R^{i}{}_{j}=2\frac{\partial G^{i}}{\partial x^{j}}-\frac{\partial^{2}G^{i}}{\partial x^{k}\partial y^{j}}y^{k}+2G^{k}\frac{\partial^{2}G^{i}}{\partial y^{k}
\partial y^{j}}-\frac{\partial G^{i}}{\partial y^{k}}\frac{\partial G^{k}}{\partial y^{j}}.$$

The value as follows
\begin{eqnarray*}
   K(P,y):=\frac{g_{y}(R_{y}(u),u)}{g_{y}(y,y)g_{y}(u,u)-[g_{y}(y,u)]^{2}}
\end{eqnarray*}
is called the flag curvature of the flag plane $P=\textrm{span}\{y,u\}\subset T_{x}M$ along the direction $y$.
When $F$ is Riemannian, $K(P,y)=K(P)$ is independent of $y\in P$ and it
is just the sectional curvature of $P$ in Riemann geometry. $F$
is said to be of constant flag curvature if for any $y\in T_x M$ , the
flag curvature $K(P,y)=K$ is a constant, that is equivalent to the following system of equations in
a local coordinate system $(x^{i},y^{i})$ in $TM$ ,
\begin{eqnarray*}
  R^{i}{}_{j} =K F^{2}(\delta^{i}{}_{j}-F^{-1}F_{y^{j}}y^{i}).
\end{eqnarray*}

On the other hand, a Finsler metric $F$ on a manifold $M$ is said to be locally projectively flat if at any point, there is a local coordinate system $(x^{i})$ in which the geodesics are straight lines as point sets. In this case, the spray coefficients are in the form $G^{i}=Py^{i}$, where $P=P(x,y)$ given by $P=\frac{F_{x^{k}}y^{k}}{2F}$ is called the projective factor of $F$. For a projectively flat Finsler metric $F$, the flag curvature is given by
\begin{eqnarray}\label{K}
K=\frac{P^2-P_{x^k}y^k}{F^2}.
\end{eqnarray}

By definition, a general $(\a,\b)$-metric is given by (\ref{generalab}) where $\phi=\phi(b^2,s)$ is a smooth function defined on the domain $|s|\leq b<b_o$ for some positive number (maybe infinity) $b_o$, $\alpha$ is a Riemannian metric  and  $\beta $ is a $1$-form  with $b<b_o$. When $n\geq3$, $F= \a\phi(b^2,\frac{\b}{\a})$ is a regular Finsler metric for any $\alpha$ and $\beta$ with $ b<b_o$
if and only if $\p(b^2,s)$ satisfies
\begin{eqnarray*}
\p-s\pt>0,\quad\p-s\pt+(b^2-s^2)\ptt>0,\qquad|s| \leq b < b_o.
\end{eqnarray*}

Let $\alpha=\sqrt{a_{ij}(x)y^iy^j}$  and $\beta= b_i(x)y^i$.
Denote the coefficients of the covariant derivative of
$\b$ with respect to $\a$ by $b_{i|j}$, and let
\begin{eqnarray*}
&r_{ij}=\frac{1}{2}(b_{i|j}+b_{j|i}),~s_{ij}=\frac{1}{2}(b_{i|j}-b_{j|i}),
~r_{00}=r_{ij}y^iy^j,~s^i{}_0=a^{ij}s_{jk}y^k,&\\
&r_i=b^jr_{ji},~s_i=b^js_{ji},~r_0=r_iy^i,~s_0=s_iy^i,~r^i=a^{ij}r_j,~s^i=a^{ij}s_j,~r=b^ir_i,&
\end{eqnarray*}
where $(a^{ij}):=(a_{ij})^{-1}$ and $b^{i}:=a^{ij}b_{j}$. It is easy to see that $\b$ is closed if and only if $s_{ij}=0$.

\begin{lemma}\cite{yct-zhm-onan}
The geodesic coefficients $G^{i}$ of a general $(\alpha,\beta)$-metric $F=\a\phi(b^2,\frac{\b}{\a})$ are given by
\begin{eqnarray}\label{Gi}
G^i&=&{}^\a G^i+\a Q s^i{}_0+\left\{\Theta(-2\a Q s_0+r_{00}+2\a^2
R r)+\a\Omega(r_0+s_0)\right\}\frac{y^i}{\a}\nonumber\\
&&+\left\{\Psi(-2\a Q s_0+r_{00}+2\a^2 R
r)+\a\Pi(r_0+s_0)\right\}b^i -\a^2 R(r^i+s^i),
\end{eqnarray}
where ${}^\a G^i$ are the geodesic coefficients of $\alpha$, and
\begin{eqnarray*}
&Q=\frac{\pt}{\p-s\pt},\quad R=\frac{\po}{\p-s\pt},\quad \Omega=\frac{2\po}{\p}-\frac{s\p+(b^2-s^2)\pt}{\p}\Pi,&\\
&\Theta=\frac{(\p-s\pt)\pt-s\p\ptt}{2\p\big(\p-s\pt+(b^2-s^2)\ptt\big)},
\quad\Psi=\frac{\ptt}{2\big(\p-s\pt+(b^2-s^2)\ptt\big)},\quad
\Pi=\frac{(\p-s\pt)\pot-s\po\ptt}{(\p-s\pt)\big(\p-s\pt+(b^2-s^2)\ptt\big)}.&
\end{eqnarray*}
\end{lemma}
Note that $\phi_1$ means the derivation of $\phi$ with respect to the first variable $b^2$.

Finally, it is known that if the geodesic spray coefficients of a Finsler metric $F$ are given by
$$G^i={}^\a G^i+Q^i,$$
then the Riemann curvature tensor of $F$ are related to that of $\a$ and given by
\begin{eqnarray}\label{riccichange}
R^i{}_j={}^\a R^i{}_j+2Q^i{}_{|j}-y^mQ^i{}_{|m.j}+2Q^mQ^i{}_{.m.j}-Q^i{}_{.m}Q^m{}_{.j},
\end{eqnarray}
where ``${}_|$" and  ``$.$" denote the horizontal covariant derivative and vertical covariant derivative with respect to $\a$ respectively, i.e., $*_{.i}=\pp{*}{y^i}$.

\section{Constant sectional curvature Riemannian metrics and their conformal $1$-forms}
According to \cite{yct-dhfp}, if $\a$ and $\b$ satisfy (\ref{condition}), then there is a local coordinate system in which
\begin{eqnarray}\label{ab}
\a=\f{\sqrt{(1+\mu|x|^2)|y|^2-\mu\langle x,y\rangle^2}}{1+\mu|x|^2},\qquad
\b=\f{\lambda\langle x,y\rangle+(1+\mu|x|^2)\langle a,y\rangle-\mu\langle a,x\rangle\langle x,y\rangle}{(1+\mu|x|^2)^\frac{3}{2}}.
\end{eqnarray}
In this case,
\begin{eqnarray}\label{bija}
\bij=\f{\lambda-\mu\langle a,x\rangle}{\sqrt{1+\mu|x|^2}}\aij.
\end{eqnarray}
One can check directly
\begin{eqnarray}\label{km}
c^2(x)=\lambda^2+\mu|a|^2-\mu b^2.
\end{eqnarray}
Hence, we immediately have
\begin{lemma}\label{c(x)}
If $\a$ and $\b$ satisfy (\ref{condition}), then
$$c^2=\kappa-\mu b^2$$
for some constant $\kappa$.
\end{lemma}

The constant $\kappa$ has specific geometric meaning. In order to see it, we need some discussions on wrap product.

Because $\b$ is closed, we can assume locally $\b=\ud f\neq0$ for some smooth function $f(x)$. It is easy to see that the condition $\bij=c\aij$ is equivalent to $\mathrm{Hess}_{\a}f=c\a^2$. According to P. Petersen's result, in this case
\begin{eqnarray}\label{wa}
\a^2=\ud t\otimes\ud t+h^2(t)\breve{\alpha}^2
\end{eqnarray}
must be locally a warped product metric on the manifold $M=\RR\times{\breve M}$, where $\breve M$ is an $(n-1)$-dimensional manifold equipped with the Riemannian metric $\breve\alpha$. Moreover, the function $f$ depends only on the parameter $t$ of $\RR$ and $h(t)=f'(t)$\cite{pete}.

Let $x^1=t$ and $\{x^a\}_{a=2}^n$ be a local coordinate system on $\breve M$, then the Riemann curvature tensor of $\a$ is determined by\cite{bao-cheng-shen}
\begin{eqnarray*}
R^1{}_j=-\f{h''}{h}(\a^2\delta^1{}_j-y^1y_j),\qquad
R^a{}_c={}^{\breve{\alpha}}R^a{}_c-(h')^2(\breve\a^2\delta^a{}_c-\breve y^a\breve y_c)-\f{h''}{h}(y^1)^2\delta^a{}_c,
\end{eqnarray*}
where $\breve y^a=y^a$ and $\breve y_c=\breve a_{ac}\breve y^a$. Hence, if $\a$ is of constant sectional curvature, then combining with the first equality of (\ref{condition}) and the above two equalities we obtain $h''+\mu h=0$
and
\begin{eqnarray}
{}^{\breve{\alpha}}R^a{}_c=\left[\mu h^2+(h')^2\right](\breve\alpha^2\delta^a{}_c-\breve y^a\breve y_c).\label{ricbreve}
\end{eqnarray}
On the other hand, $\b=\ud f=h(t)\,\ud t$ by assumption and hence $b^2=h^2$. Direct computations show that $\bij=h'(t)\aij$,
so by Lemma \ref{c(x)} we have
$$\mu h^2+(h')^2=\mu b^2+c^2=\kappa.$$

\begin{lemma}
The Riemannian metric $\breve\a$ in (\ref{wa}) is of constant sectional curvature $\kappa$.
\end{lemma}

Next, we will show that the case when $\kappa\neq0$ and $\mu\neq0$ could be reduce to the case $\mu=0$. We need some special metrical deformations for $\a$ and $\b$, one can see \cite{yct-dhfp} for details about these deformations.

\begin{lemma}\label{lemma}
When $\kappa\neq0$, $\mu\neq0$ and $\kappa-\mu b^2 >0$, define $\ba$ and $\bb$ by
\begin{eqnarray*}
\ba^2=\f{|\mu|}{\kappa-\mu b^2}\left(\a^2+\f{\mu}{\kappa-\mu b^2}\b^2\right),
\quad\bb=\f{|\mu|^{3/2}}{(\kappa-\mu b^2)^\frac{3}{2}}\b, \label{deformationab}
\end{eqnarray*}
then
$${}^{\bar\a}R^i{}_j=0,\qquad \bbij=\pm\sqrt{|\mu|}\baij.$$
In this case,
$$(\kappa-\mu b^2)(\kappa^{-1}+\mu^{-1}\bar b^2)=1,$$
and the reversed deformations are given by
\begin{eqnarray*}
\a^2=\f{|\mu|^{-1}}{\kappa^{-1}+\mu^{-1}\bar{b}^2}\left(\ba^2-\frac{\mu^{-1}}{\kappa^{-1}+\mu^{-1}b^2}\bb^2\right),
\quad\b=\f{|\mu|^{-3/2}}{(\kappa^{-1}+\mu^{-1}\bar b^2)^\frac{3}{2}}\bb.
\end{eqnarray*}
\end{lemma}
\begin{proof}
It is easy to see that
$$\bG=\G+\f{c(x)\mu}{\kappa-\mu b^2}\b y^i.$$
Let $\bar Q^i=\f{c(x)\mu}{\kappa-\mu b^2}\b y^i$, then
\begin{eqnarray*}
\bar Q^i{}_{|j}&=&\mu\left(y^iy_j+\f{\mu}{\kappa-\mu b^2}\b y^ib_j\right),\\
y^k\bar Q^i{}_{|k.j}&=&\mu\left(\a^2\delta^i{}_j+y^iy_j+\f{\mu}{\kappa-\mu b^2}\b^2\delta^i{}_j
+\f{\mu}{\kappa-\mu b^2}\b y^ib_j\right),\\
\bar Q^i{}_{.k}\bar Q^k{}_{.j}&=&\f{\mu^{2}}{\kappa-\mu b^2}(\beta^2\delta^i{}_j+3\b y^ib_j),\\
\bar Q^k\bar Q^i{}_{.k.j}&=&\f{\mu^{2}}{\kappa-\mu b^2}(\beta^2\delta^i{}_j+\b y^ib_j),
\end{eqnarray*}
where $y_i=\aij y^j$. So by (\ref{riccichange}) we have
$${}^{\ba}R^i{}_j={}^\a R^i{}_j-\mu(\a^2\delta^i{}_j-y^iy_j)=0.$$
On the other hand, direct computations show that
\begin{eqnarray*}
\bbij&=&\f{c |\mu|^{3/2}}{\left(\kappa-\mu b^2\right)^\frac{3}{2}}\left(\aij+\f{\mu}{{\kappa}-\mu b^2}
\bi\bj\right)=\pm\sqrt{|\mu|}\baij.
\end{eqnarray*}
\end{proof}

Notice that when $\kappa<0$, $\ba^2$ is a pseudo-Riemannian metric of signature $(n-1,1)$, because it is positive definite on the hyperplane $\beta=0$ and negative when $y^i=b^i$. In particular, the norm of $\bb$ with respect to $\ba$ is negative, i.e., $\bar b^2<0$.

$\kappa=0$~(in this case $\mu$ must be negative by Lemma \ref{c(x)}) is a very special case, because the metrical deformation given below is irreversible.
\begin{lemma}\label{lemmaewe}
When $\mu<0$ and $\kappa=0$, define $\ba$ and $\bb$ by
\begin{eqnarray*}
\ba:=\f{\a}{b},
\quad\bb:=\f{\b}{b^2},
\end{eqnarray*}
then
\begin{equation*}
{}^{\bar\a}R^i{}_j=0,\qquad \bbij=0.
\end{equation*}
In this case, $\bar b=1$.
\end{lemma}
\begin{proof}
It is easy to see that
$$\bG=\G-\f{c}{b^2}\b y^i+\f{c}{2b^2}\a^2b^i,$$
where $c=c(x)$ is a scalar function with $c^2 = -\mu b^2$.
Let $\bar Q^i=-\f{c}{b^2}\b y^i+\f{c}{2b^2}\a^2b^i.$ Then
\begin{eqnarray*}
\bar Q^i{}_{|j}&=&-\f{\mu}{2}\left(\a^2\delta^{i}{}_{j}-2y^{i}y_{j}-\frac{\a^2}{b^{2}}b^{i}b_{j}+\f{2}{b^2}\b b_{j}y^{i}\right),\\
y^k\bar Q^i{}_{|k.j}&=&\mu \left\{\left(\a^2-\frac{\b^2}{b^{2^{}}}\right)\delta^{i}{}_{j}-\f{\b}{b^2}(b^{i}y_{j}-b_{j}y^{i})\right\},\\
\bar Q^i{}_{.k}\bar Q^k{}_{.j}&=&-\frac{\mu}{b^{2}}\left(\b^2\delta^{i}{}_{j}+3\b b_{j}y^{i}-\b b^{i}y_{j}-\a^2b^{i}b_{j}-b^2 y^{i}y_{j}\right),\\
\bar Q^k\bar Q^i{}_{.k.j}&=&-\frac{\mu}{b^{2}}\left\{\left(\b^2-\frac{\a^2 b^{2}}{2}\right)\delta^{i}{}_{j}+\b(b_{j}y^{i}-y_{j}b^{i})-\frac{\a^2}{2}(b^{i}b_{j}-b_{j}b^{i})\right\},
\end{eqnarray*}
where $y_{i}=a_{ij}y^{j}$.
So by (\ref{riccichange}) we have
$${}^{\ba}R^i{}_j={}^\a R^i{}_j-\mu(\a^2\delta^i{}_j-y^iy_j)=0.$$
On the other hand, direct computations show that $\bbij=0$.
\end{proof}

\section{Proof of Theorem \ref{main1}}
\begin{proof}[Proof of Theorem \ref{main1}]
Because $\a$ is of constant sectional curvature, $\a$ must be locally projectively flat due to Beltrami's theorem. Hence, there is a local coordinate system such that ${}^\a G^i=\theta y^i$. By (\ref{Gi}), the spray coefficients $G^i$ of $F$ is given by $G^i=(\theta+c\a\psi)y^i.$

It is easy to see that
$$\a_{x^k}y^k=2\a\theta,\qquad\b_{x^k}y^k=c\a^2+2\b\theta,\qquad c_{x^k}y^k=-\mu\b,$$
where the third equality is based on Lemma \ref{c(x)}. Then
$$(c\a\psi)_{x^k}y^k=\a^2[-\mu s\psi+c^2(\psi_2+2s\psi_1)]+2c\a\theta\psi.$$
So by (\ref{K}) we have
\begin{eqnarray*}
K&=&\frac{(\theta+c\a\psi)^2-\theta_{x^k}y^k-(c\a\psi)_{x^k}y^k}{F^2}\\
&=&\f{\{\theta^2-\theta_{x^k}y^k\}+\a^2\{\mu s\psi+c^2[\psi^2-(\psi_2+2s\psi_1)]\}}{F^2}\\
&=&\f{\mu+\{\mu s\psi+c^2[\psi^2-(\psi_2+2s\psi_1)]\}}{\p^2}.
\end{eqnarray*}
Here we use the fact that $\a$ is projectively flat and hence $\theta^2-\theta_{x^k}y^k=\mu\a^2$ by (\ref{K}).
\end{proof}

In the rest of this paper, we will determine all the general $\ab$-metrics with constant flag curvature under our assumption. There are four different cases below,
\begin{enumerate}[(a)]
\item $\kappa=0$ and $\mu=0$;
\item $\kappa\neq0$ and $\mu=0$;
\item $\kappa\neq0$ and $\mu\neq0$;
\item $\kappa=0$ and $\mu\neq0$.
\end{enumerate}

As we have pointed out in Section 1, the case (a) is trivial and will not be discussed. The case (b) will be discussed in Section 6 and Section 8.

The case (c) can be reduced to the case (b) and hence it is not necessary to be discussed specially. The reason is below.
If $F=\a\phi(b^2,\frac{\b}{\a})$ is a general $\ab$-metric satisfying Theorem \ref{main1} with $\kappa\neq0$ and $\mu\neq0$, then after deformations in Lemma \ref{lemma}, the new data $(\ba,\bb)$ satisfies the condition (\ref{conditions}) with $\kappa\neq0$ and $\mu=0$. As a result, $F$ can be reexpressed as a new form $F=\ba\bar\phi(\bar b^2,\frac{\bb}{\ba})$. In \cite{szm-yct-oefm}, we have proved that if $\phi$ satisfies Equation (\ref{pde}) and (\ref{pde2}), then $\bar\phi$ also satisfies Equation (\ref{pde}) and (\ref{pde2}) with $\bar\kappa=|\mu|$ and $\bar\mu=0$. That is to say, all the solutions provided by (c) are included naturally by in (b).

On the other hand, (d) is intrinsically different from (b). Although a data $\ab$ with $\kappa\neq0$ and $\mu=0$ can turn to be a new data $(\bar\a,\bar\b)$ with $\bar\mu=0$ after deformations in Lemma \ref{lemmaewe}, (d) can not be reduced to (b) like (c). The key point is that the deformations in Lemma \ref{lemmaewe} is irreversible. The case (d) will be discussed in Section 7 and Section 9.

\section{Solutions of Equations (\ref{pde}) and (\ref{pde2}) in general case}
\begin{lemma}
The solutions of Equation (\ref{pde}) are given by
$$\p(b^2,s)=f(b^2-s^2)+2s\int_0^sf'(b^2-\sigma^2)\,\ud\sigma+g(b^2)s,$$
where $f$ and $g$ are two arbitrary smooth functions.
\end{lemma}
\begin{proof}
Make a change of variables as
\begin{eqnarray}\label{uv}
u=b^2-s^2,\qquad v=s,
\end{eqnarray}
then $b^2=u+v^2$, $s=v$. Because
$$\pp{}{v}(\p-s\pt)=(\p-s\pt)'_1\cdot(2s)+(\p-s\pt)'_2=2s(\po-s\pot)-s\ptt=0,$$
there exists a smooth function $f(u)$ such that~$\p-s\pt=f(b^2-s^2)$. Let $\phi=s\varphi$, then we have $-s^{2}\varphi_{2}=f(b^2-s^2)$. Thus
$$\varphi=\frac{1}{s}f(b^{2}-s^{2})+2\int_0^sf'(b^2-\sigma^2)d\sigma+g(b^{2}),$$
where $g$ is a smooth function.
\end{proof}

In our problem, the function $\phi(b^2,s)$ is always positive. Using the change of variables (\ref{uv}), Equation (\ref{pde2}) can be reexpressed simpler as follows,
\begin{eqnarray}\label{pde5}
[\kappa-\mu(u+v^2)]\left(\f{1}{\sqrt{\p}}\right)_{vv}-\mu v\left(\f{1}{\sqrt{\p}}\right)_v+\mu\left(\f{1}{\sqrt{\p}}\right)-K\left(\f{1}{\sqrt{\p}}\right)^{-3}=0.
\end{eqnarray}

According to the Equation 24 of Section 2.9.2 in \cite{po}, if we set $\xi=\int\f{\ud v}{\sqrt{\kappa-\mu(u+v^2)}}$, then Equation (\ref{pde5}) becomes
$$\left(\f{1}{\sqrt{\p}}\right)_{\xi\xi}+\mu\left(\f{1}{\sqrt{\p}}\right)-K\left(\f{1}{\sqrt{\p}}\right)^{-3}=0.$$
Hence, one can obtain all the positive solutions of Equation (\ref{pde2}) by solving the above equation directly.

\section{Solutions of Equations (\ref{pde}) and (\ref{pde2}) when $\kappa\neq0$ and $\mu=0$}
If $\kappa\neq0$ and $\mu=0$, Equation (\ref{pde5}) becomes
\begin{eqnarray}\label{eqn01}
\left(\frac{1}{\sqrt{\p}}\right)_{vv}=\sigma\left(\frac{1}{\sqrt{\p}}\right)^{-3},
\end{eqnarray}
where $\sigma:=\frac{K}{\kappa}$. This equation had been solved in \cite{szm-yct-oefm}.

\begin{lemma}\cite{szm-yct-oefm}\label{lemmaa}
The non-constant solutions of Equation (\ref{eqn01}) are given by
\begin{eqnarray*}\label{sol01}
\p(u,v)=\frac{1}{p(u)\pm2\sqrt{-\sigma}v}
\end{eqnarray*}
or
\begin{eqnarray*}\label{sol02}
\p(u,v)=\frac{q(u)}{(p(u)+q(u)v)^2+\sigma},
\end{eqnarray*}
where $p(u)$ and $q(u)$ are two arbitrary smooth functions.
\end{lemma}

\begin{lemma}\cite{szm-yct-oefm}\label{lemmab}
When $\mu=0$ and $\kappa\neq0$, the non-constant solutions of Equations (\ref{pde}) and (\ref{pde2}) are given by
\begin{eqnarray*}\label{sol03}
\p(b^2,s)=\f{1}{2\sqrt{-\sigma}}\cdot\f{1}{\pm\sqrt{C-b^2+s^2}\pm s}
\end{eqnarray*}
or
\begin{eqnarray*}\label{sol05}
\p(b^2,s)=\f{q(u)}{q^2(u)(Dq(u)+v)^2+\sigma},
\end{eqnarray*}
where $\sigma=\frac{K}{\kappa}$, $u:=b^2-s^2$ and $v=s$, the function $q(u)\neq0$ is determined by the following equation
\begin{eqnarray*}
D^2q^4+(u-C)q^2-\sigma=0,
\end{eqnarray*}
where $C$ and $D$ are both constant numbers.
\end{lemma}

\section{Solutions of Equations (\ref{pde}) and (\ref{pde2}) when $\kappa=0$ and $\mu\neq0$}
If $\kappa=0$ and $\mu\neq0$, Equation (\ref{pde5}) is reduced to the following form
\begin{eqnarray}\label{pde6}
(u+v^{2})f_{vv}+v f_{v}-f-\tau f^{-3}=0,
\end{eqnarray}
where $f:=\frac{1}{\sqrt{\phi}}$ and $\tau:=-\frac{K}{\mu}$.
\begin{lemma}\label{lemmac}
The non-constant solutions of Equation (\ref{pde6}) are given by
\begin{eqnarray}\label{solution5}
\phi(u,v)=\frac{2q(u)(\sqrt{u+v^{2}}\pm v)^{2}}{[q(u)(\sqrt{u+v^{2}}\pm v)^{2}+p(u)]^2+\tau}.
\end{eqnarray}
where $p(u)$ and $q(u)$ are two arbitrary functions.
\end{lemma}
\begin{proof}
Regard Equation (\ref{pde6}) as an ODE of $v$. If $f_{v}=0$, then $f$ must be a constant. If $f_{v}\neq 0$, then multiplying the both sides of (\ref{pde6}) by $f_{v}$ and integrating with respect to $v$ yields
\begin{eqnarray*}
(u+v^{2})(f_{v})^{2}=2\int(f+\tau f^{-3})\,\ud f=f^{2}-\tau f^{-2}-2p(u),
\end{eqnarray*}
where $p(u)$ is an arbitrary function of $u$. Since $u+v^{2}>0$ in our problem, by the above equality we have
 \begin{eqnarray*}
 \frac{\ud f}{\sqrt{f^{2}-\tau f^{-2}-2p(u)}}=\pm \frac{\ud v}{\sqrt{u+v^{2}}}.
 \end{eqnarray*}
So
 \begin{eqnarray*}
 f^{2}+\sqrt{f^{4}-2p(u)f^{2}-\tau}=q(u)(\sqrt{u+v^{2}}\pm v)^{2}+p(u),
 \end{eqnarray*}
 where $q(u)$ is an arbitrary function of $u$. Hence, $\phi$ is given by (\ref{solution5}).
\end{proof}

\begin{lemma}\label{lemmad}
When $\kappa=0$ and $\mu\neq0$, the non-constant solutions of Equation (\ref{pde}) and (\ref{pde2}) are given by (\ref{solution5}), where $p(u)$ and $q(u)$ satisfy an ODE system as follows:
\begin{eqnarray}
&uq^2p'+(p^2+\tau)q'=0,&\label{lem1}\\
&qp'-2pq'-uqq'-2q^2=0.&\label{lem2}
\end{eqnarray}
\end{lemma}
\begin{proof}
Using the change of variables (\ref{uv}), Equation (\ref{pde}) becomes
\begin{eqnarray*}
\p_{vv}-2v\p_{uv}-4\p_u=0.
\end{eqnarray*}
When $\phi(u,v)=\frac{2q(u)(\sqrt{u+v^{2}}- v)^{2}}{[q(u)(\sqrt{u+v^{2}}- v)^{2}+p(u)]^2+\tau}$, with the help of Maple we know that the above equation is equivalent to the following equation
\begin{eqnarray}\label{pquv}
A_{6}(u)v^{6}+A_{4}(u)v^{4}+A_{2}(u)v^{2}+A_{0}(u)+\sqrt{u+v^{2}}\left\{A_{5}(u)v^{5}+A_{3}(u)v^{3}+A_{1}(u) v\right\}=0,
\end{eqnarray}
where
\begin{eqnarray*}
A_{6}(u)&=&-A_{5}(u)=-32 q^{3}(q p'-2p q'-u q q'-2q^{2}),\\
A_{4}(u) &=& \frac{3}{2}u A_{6}(u)+M,\\
A_{3}(u) &=& u A_{5}-M,\\
A_{2}(u) &=& \frac{3}{16q^2}(3u^{2}q^2-p^2-\tau)A_{6}+\frac{1}{2q}(2uq+p)M,\\
A_{1}(u) &=& \frac{3}{16q^2}(u^{2}q^2-p^2-\tau)A_{5}-\frac{1}{2q}(uq+p)M,\\
A_{0}(u) &=& \frac{1}{32q^3}[u^{3}q^3-(3uq+2p)(p^2+\tau)]A_{6}+\frac{1}{24 q^{2}}(3u^2q^2+6upq+3p^2-\tau)M,
\end{eqnarray*}
and $M:=24q^{2}[u q^{2}p'+(p^{2}+\tau)q']$.

Since $\sqrt{u+v^2}$ is irrational with respect to $v$ and the remaining parts of Equation (\ref{pquv}) are rational, Equation (\ref{pquv}) holds if and only if
$$A_{6}(u)v^{6}+A_{4}(u)v^{4}+A_{2}(u)v^{2}+A_{0}(u)=0,\qquad A_{5}(u)v^{5}+A_{3}(u)v^{3}+A_{1}(u) v=0.$$
As a result, $A_i(u)=0$ for $1\leq i\leq6$, which are equivalent to Equations (\ref{lem1}) and (\ref{lem2}).

We can obtain the same equations similarly when $\phi(u,v)=\frac{2q(u)(\sqrt{u+v^{2}}+ v)^{2}}{[q(u)(\sqrt{u+v^{2}}+ v)^{2}+p(u)]^2+\tau}$.
\end{proof}

\begin{lemma}\label{leme}
The solutions of Equations (\ref{lem1}) and (\ref{lem2}) with $q(u)\not\equiv0$ are given by
\begin{eqnarray}\label{solutionpone}
p(u)=\pm\sqrt{-\tau},\qquad q(u)=\pm\f{(C\pm\sqrt{C^2+8pu})^2}{4u^2}
\end{eqnarray}
or
\begin{eqnarray}
p(u)&=&\pm\sqrt{\f{-(C^2-D)\tau-C(C\tau-2u)\pm\sqrt{D(C\tau-2u)^2-D(C^2-D)\tau^2}}{2(C^2-D)}},\label{ppp}\\
q(u)&=&\f{p^2+\tau-upp'\pm\sqrt{(p^2+\tau-upp')^2-(p^2+\tau)u^2p'}}{u^2p'},\label{qqq}
\end{eqnarray}
where $C$ and $D$ are constants.
\end{lemma}
\begin{proof}
(\ref{lem1})$\times(uq+2p)+$(\ref{lem2})$\times(p^2+\tau)$ yields
\begin{eqnarray}\label{q2}
u^2p'q^2-2(p^2+\tau-upp')q+(p^2+\tau)p'=0.
\end{eqnarray}

When $p'=0$, then $p=\pm\sqrt{-\tau}$ by (\ref{lem1}). In this case, (\ref{lem2}) is equivalent to
$$\left(u\sqrt[4]{q^2}\right)'=2p\left(\f{1}{\sqrt[4]{q^{2}}}\right)',$$
so
$$u\sqrt[4]{q^2}=2p\cdot\f{1}{\sqrt[4]{q^2}}+C$$
for some constant $C$, which leads to the solutions (\ref{solutionpone}).

When $p'\neq0$, then by (\ref{q2})
$$q=\f{p^2+\tau}{p^2+\tau-upp'\pm\sqrt{(p^2+\tau-upp')^2-(p^2+\tau)u^2p'}}.$$
Putting the above equality into (\ref{lem1}) yields
\begin{eqnarray*}
(p^2+\tau)^2p''+(p^2+\tau)p(p')^2+2\tau u(p')^3=0.
\end{eqnarray*}
Regard $u$ as the function of $p$, then the above equation turns to be
$$u''-\f{p}{p^2+\tau}u'-\f{2\tau}{(p^2+\tau)^2}u=0,$$
and its solutions are given below,
$$u=C(p^2+\tau)\pm\sqrt{D}p\sqrt{p^2+\tau},$$
where $C$ and $D$ are constants, and hence $p$ can be solved and given by (\ref{ppp}). Notice that the constant $D$ here can be negative.
\end{proof}

\section{Proof of Theorem \ref{main2} and some regular examples}
\begin{proof}[Proof of Theorem \ref{main2}]
It is true by Theorem \ref{main1} and Lemma \ref{lemmab}.
\end{proof}

Example \ref{ex1}-\ref{ex4} show four typical kinds of regular general $\ab$-metrics in our problem, and Example \ref{ex5} shows that we can also give the analytic expressions in the case $\kappa\neq0$ and $\mu\neq0$.
\begin{example}\label{ex1}
Take $\mu=0$, $\lambda=1$ in (\ref{ab}) and $\sigma=-\frac{1}{4}$, $C=1$ in (\ref{solution1}), then
$$\phi(b^2,s)=\frac{\sqrt{1-b^2+s^2}}{1-b^2}\pm\frac{s}{1-b^2},$$
and the corresponding general $\ab$-metrics
$$F=\f{\sqrt{(1-\xx-2\langle a,x\rangle-|a|^2)\yy+(\xy+\langle a,x\rangle)^2}}{1-\xx-2\langle a,x\rangle-|a|^2}
\pm\f{\xy+\langle a,x\rangle}{1-\xx-2\langle a,x\rangle-|a|^2}$$
are locally projectively flat with constant flag curvature $K=-\frac{1}{4}$. Actually, they are just the generalized Funk metrics (\ref{funk}) expressed in some other local coordinate system.
\end{example}

\begin{example}\label{ex2}
Take $\mu=0$, $\lambda=1$ in (\ref{ab}) and $\sigma=0$, $C=D=1$ in (\ref{solution2}), then parts of the solutions of (\ref{solution2}) are given by
$$\phi(b^2,s)=\frac{1}{\sqrt{1-b^2+s^2}(\sqrt{1-b^2+s^2}\pm s)^2}=\f{(\sqrt{1-b^2+s^2}\mp s)^2}{(1-b^2)^2\sqrt{1-b^2+s^2}},$$
and the corresponding general $\ab$-metrics
$$F=\frac{\{\sqrt{(1-\xx-2\langle a,x\rangle-|a|^2)\yy+(\xy+\langle a,y\rangle)^2}\mp(\xy+\langle a,y\rangle)\}^2}{(1-\xx-2\langle a,x\rangle-|a|^2)^2\sqrt{(1-\xx-2\langle a,x\rangle-|a|^2)\yy+(\xy+\langle a,y\rangle)^2}}$$
are locally projectively flat with constant flag curvature $K=0$. Actually, they are just the generalized Berwald's metrics (\ref{berwald}) expressed in some other local coordinate system.
\end{example}

\begin{example}\label{ex3}
Take $\mu=0$, $\lambda=1$ in (\ref{ab}) and $\sigma=1$, $C=D=1$ in (\ref{solution2}), then one solution of (\ref{solution2}) is given by
$$\phi(b^2,s)=\Re\frac{1}{\sqrt{1+2i+b^2-s^2}+is},$$
and the corresponding general $\ab$-metrics
$$F=\Re\frac{\yy}{\sqrt{(1+2i+\xx+2\langle a,x\rangle+|a|^2)\yy-(\xy+\langle a,y\rangle)^2}+i(\xy+\langle a,y\rangle)}$$
are locally projectively flat with constant flag curvature $K=1$. They are parts of Bryant's metrics\cite{Br2,yct-zhm-onan}.
\end{example}

\begin{example}\label{ex4}
Take $\mu=0$, $\lambda=1$ in (\ref{ab}) and $\sigma=-1$, $C=\frac{1}{2}\left(1+\frac{1}{\varepsilon^2}\right)$, $D=\frac{1}{4}\left(1-\frac{1}{\varepsilon^2}\right)$ where $0<|\varepsilon|<1$ in (\ref{solution2}), then part of the solutions of (\ref{solution2}) is given by
$$\phi(b^2,s)=\f{1}{2}\left\{\frac{\sqrt{1-b^2+s^2}+s}{1-b^2}
-\frac{\varepsilon\sqrt{1-\varepsilon^2b^2+\varepsilon^2s^2}+\varepsilon^2s}{1-\varepsilon b^2}\right\},$$
and the corresponding general $\ab$-metrics
\begin{eqnarray*}
F&=&\f{1}{2}\Bigg\{\frac{\sqrt{(1-\xx-2\langle a,x\rangle-|a|^2)\yy+(\xy+\langle a,y\rangle)^2}+\xy+\langle a,y\rangle}{1-\xx-2\langle a,x\rangle-|a|^2}\\
&&-\frac{\varepsilon\sqrt{[1-\varepsilon^2(\xx+2\langle a,x\rangle+|a|^2)]\yy+\varepsilon^2(\xy+\langle a,y\rangle)^2}+\varepsilon^2(\xy+\langle a,y\rangle)}{1-\varepsilon^2(\xx+2\langle a,x\rangle+|a|^2)}\Bigg\}
\end{eqnarray*}
are locally projectively flat with constant flag curvature $K=-1$. They include Shen's metrics of \cite{szm-pffm} as (39) in it.
\end{example}

\begin{example}\label{ex5}
Let $\a$ and $\b$ be data satisfying (\ref{conditions}) with $\mu\neq0$ and $\kappa\neq0$. According to Lemma 7.1 in \cite{szm-yct-oefm}, the following function
\begin{eqnarray*}
\phi(b^2,s):=\f{\sqrt{|\mu|}\sqrt{\kappa-\mu b^2+\mu s^2}}{\kappa-\mu b^2}\bar\phi\left(\f{\mu b^2}{\kappa-\mu b^2}\frac{\mu}{\kappa},\; \f{|\mu| s}{\sqrt{\kappa-\mu b^2}\sqrt{\kappa-\mu b^2+\mu s^2}}\right)
\end{eqnarray*}
satisfies (\ref{pde}) and (\ref{pde2}) if and only if $\bar\phi(b^2,s)$ is one of the functions given in (\ref{solution1}) or (\ref{solution2}). Hence, by Theorem \ref{main1} we know that the corresponding general $\ab$-metric $F=\a\phi(b^2,\frac{\b}{\a})$ is locally projectively flat with constant flag curvature $K$. By the arguments in Section 3, these metrics are just the metrics in Theorem \ref{main2} given in a different form.
\end{example}

\section{Proof of Theorem \ref{main3} and some non-regular examples}
\begin{proof}[Proof of Theorem \ref{main3}]
It is true by Theorem \ref{main1},Lemma \ref{lemmad} and Lemma \ref{leme}.
\end{proof}

Let $\a$ and $\b$ are given by (\ref{ab}). According to (\ref{km}), $\a$ and $\b$ satisfy (\ref{conditions}) with $\kappa=0$ if and only if
\begin{eqnarray}\label{la}
\lambda^2+\mu|a|^2=0.
\end{eqnarray}
In this case, the length of $\b$ is given by
$$b=\f{|\lambda-\mu\langle a,x\rangle|}{\sqrt{-\mu}\cdot\sqrt{1-\mu|x|^2}}.$$

As a application of Theorem \ref{main3}, some typical general $\ab$-metrics with constant flag curvature are analytic constructed below. Note that all of them are of some singularity.
\begin{center}
$$\mathbf K=0$$
\end{center}
\begin{example}
Let $\a$ and $\b$ be data in (\ref{ab}) with an additional condition (\ref{la}) and take $\tau=0$, $p(u)=0$, $q(u)=\frac{2}{u^2}$ in (\ref{pppp}), then one solution of (\ref{solution3}) is given by
$$\phi(b^2,s)=(b+s)^2,$$
and the corresponding general $\ab$-metrics
$$F=\f{\left\{|\lambda-\mu\langle a,x\rangle|\sqrt{(1+\mu|x|^2)|y|^2-\mu\langle x,y\rangle^2}+\sqrt{-\mu}\left(\lambda\langle x,y\rangle+(1+\mu|x|^2)\langle a,y\rangle-\mu\langle a,x\rangle\langle x,y\rangle\right)\right\}^2}{-\mu(1+\mu|x|^2)^2\sqrt{(1+\mu|x|^2)|y|^2-\mu\langle x,y\rangle^2}}$$
are locally projectively flat with vanishing flag curvature.
\end{example}

\begin{example}
Let $\a$ and $\b$ be data in (\ref{ab}) with an additional condition (\ref{la}) and take $\tau=0$, $p(u)=\frac{\sqrt{u}}{2}$, $q(u)=\frac{1}{2\sqrt{u}}$ in (\ref{pp}) and (\ref{qq}), then parts of the solutions of (\ref{solution3}) are given by
$$\phi(b^2,s)=\f{\sqrt{b^2-s^2}}{b^2}$$
and the corresponding general $\ab$-metrics
$$F=\f{\sqrt{b^2\a^2-\b^2}}{b^2}$$
are locally projectively flat with vanishing flag curvature. Actually, $F$ is a positive semi-definite Riemannian metric of signature $(n-1,0)$.
\end{example}

\begin{example}
Let $\a$ and $\b$ be data in (\ref{ab}) with an additional condition (\ref{la}) and take $\tau=0$, $p(u)=c_1\frac{\sqrt{1+c_1u}}{2}$, $q(u)=\frac{\sqrt{1+c_1u}\left(1+c_2\sqrt{1+c_1u}\right)^2}{2u^2}$ where $c_1=\pm1$, $c_2=\pm1$ in (\ref{pp}) and (\ref{qq}), then parts of the solutions of (\ref{solution3}) are given by
$$\phi(b^2,s)=\f{\left\{1+c_2\sqrt{1+c_1(b^2-s^2)}\right\}^2(b+s)^2}
{\sqrt{1+c_1(b^2-s^2)}\left\{1+c_1b(b+s)+c_2\sqrt{1+c_1(b^2-s^2)}\right\}},$$
and the corresponding general $\ab$-metrics are locally projectively flat with vanishing flag curvature.
\end{example}
\newpage
\begin{center}
$$\mathbf K=-1$$
\end{center}
\begin{example}
Let $\a$ and $\b$ be data in (\ref{ab}) with an additional condition (\ref{la}) and take $\tau=-1$, $p(u)=c_1$, $q(u)=\f{2\left(1+c_2\sqrt{1+c_1u}\right)^2}{u^2}$ where $c_1=\pm1$, $c_2=\pm1$ in (\ref{pppp}), then parts of the solutions of (\ref{solution3}) are given by
$$\phi(b^2,s)=\f{(b+s)^2}{2\left\{1+c_1b(b+s)+c_2\sqrt{1+c_1(b^2-s^2)}\right\}},$$
and the corresponding general $\ab$-metrics are locally projectively flat with flag curvature $K=-1$.
\end{example}

\begin{example}
Let $\a$ and $\b$ be data in (\ref{ab}) with an additional condition (\ref{la}) and take $\tau=-1$, $p(u)=\sqrt{1+c_1u}$, $q(u)=c_1\f{1}{\sqrt{1+c_1u}+c_2}$ where $c_1=\pm1$, $c_2=\pm1$ in (\ref{pp}) and (\ref{qq}), then parts of the solutions of (\ref{solution3}) are given by
$$\phi(b^2,s)=\f{b^2-s^2}{2b\left\{b\sqrt{1+c_1(b^2-s^2)}+c_2s\right\}},$$
and the corresponding general $\ab$-metrics are locally projectively flat with flag curvature $K=-1$.
\end{example}

\begin{example}
Let $\a$ and $\b$ be data in (\ref{ab}) with an additional condition (\ref{la}) and take $\tau=-1$, $p(u)=\f{1}{\sqrt{2}}\sqrt{1+c\sqrt{1-u^2}}$, $q(u)=-\f{\sqrt{2}\sqrt{1+c\sqrt{1-u^2}}+u}{\sqrt{2}\left(1+c\sqrt{1-u^2}\right)^\frac{3}{2}}$ where $c=\pm1$ in (\ref{pp}) and (\ref{qq}), then parts of the solutions of (\ref{solution3}) are given by
$$\phi(b^2,s)=\f{2\sqrt{2}\left(\sqrt{2}\sqrt{1+c\sqrt{1-(b^2-s^2)^2}}+b^2-s^2\right)\left(1+c\sqrt{1-(b^2-s^2)^2}\right)^\frac{3}{2}
\left(b+s\right)^2}{2\left(1+c\sqrt{1-(b^2-s^2)^2}\right)^3
-\left\{\left(\sqrt{2}\sqrt{1+c\sqrt{1-(b^2-s^2)^2}}+b^2-s^2\right)\left(b+s\right)^2-\left(1+c\sqrt{1-(b^2-s^2)^2}\right)^2\right\}^2},$$
and the corresponding general $\ab$-metrics are locally projectively flat with flag curvature $K=-1$.
\end{example}

\begin{center}
$${\mathbf K=1}$$
\end{center}
\begin{example}
Let $\a$ and $\b$ be data in (\ref{ab}) with an additional condition (\ref{la}) and take $\tau=1$, $p(u)=\f{1}{\sqrt{2}}\sqrt{\sqrt{1+u^2}-1}$, $q(u)=\f{2\sqrt{\sqrt{1+u^2}-1}+\sqrt{2}u}{2\left(\sqrt{1+u^2}-1\right)^\frac{3}{2}}$ in (\ref{pp}) and (\ref{qq}), then one solution of (\ref{solution3}) is given by
$$\phi(b^2,s)=\f{2\sqrt{2}\left(\sqrt{2}\sqrt{\sqrt{1+(b^2-s^2)^2}-1}+b^2-s^2\right)\left(\sqrt{1+(b^2-s^2)^2}-1\right)^\frac{3}{2}
\left(b+s\right)^2}{2\left(\sqrt{1+(b^2-s^2)^2}-1\right)^3
+\left\{\left(\sqrt{2}\sqrt{\sqrt{1+(b^2-s^2)^2}-1}+b^2-s^2\right)\left(b+s\right)^2+\left(\sqrt{1+(b^2-s^2)^2}-1\right)^2\right\}^2},$$
and the corresponding general $\ab$-metrics are locally projectively flat with flag curvature $K=1$.
\end{example}

\noindent Changtao Yu\\
School of Mathematical Sciences, South China Normal
University, Guangzhou, 510631, P.R. China\\
aizhenli@gmail.com
\newline
\newline
\newline
\noindent Hongmei Zhu\\
College of Mathematics and Information Science, Henan Normal University, Xinxiang, 453007, P.R. China\\
zhm403@163.com

\end{document}